\documentclass{amsart}

\usepackage{amsfonts}
\usepackage{amssymb}

\newcommand{\rn}{\mathbb{R}^N}
\newcommand{\rd}{\textrm{d}}
\renewcommand{\O}{\Omega}
\renewcommand{\H}{\mathcal{H}}
\newcommand{\lqb}{L^{q}(\partial\Omega)}

\newtheorem{de}{Definition}[section]
\newtheorem{lem}[de]{Lemma}
\newtheorem{te}[de]{Theorem}

\newtheorem{re}[de]{Remark}

\title[Remarks on an optimization problem]{Remarks on an optimization problem for the $p-$Laplacian}
\author[L. Del Pezzo and J. Fern\'andez Bonder]
{Leandro M. Del Pezzo and Juli\'an Fern\'andez Bonder}

\address{Leandro M. Del Pezzo \hfill\break\indent
Departamento  de Matem\'atica, FCEyN, Universidad de Buenos Aires,
\hfill\break\indent Pabell\'on I, Ciudad Universitaria (1428),
Buenos Aires, Argentina.}

\email{{\tt ldpezzo@dm.uba.ar}}

\address{Juli\'an Fern\'andez Bonder \hfill\break\indent
Departamento  de Matem\'atica, FCEyN, Universidad de Buenos Aires,
\hfill\break\indent Pabell\'on I, Ciudad Universitaria (1428),
Buenos Aires, Argentina.}

\email{{\tt jfbonder@dm.uba.ar}\hfill\break\indent {\it Web page:}
{\tt http://mate.dm.uba.ar/$\sim$jfbonder}}

\thanks{Supported by Universidad de Buenos Aires under grant X078,
by ANPCyT PICT No. 2006-290 and CONICET (Argentina) PIP 5478/1438.
J. Fern\'andez Bonder is a member of CONICET. Leandro Del Pezzo is a fellow of CONICET}

\begin{document}

\begin{abstract}
In this note we give some remarks and improvements on a recent paper of us \cite{nos} about an optimization problem for the $p-$Laplace operator that were motivated by some discussion the authors had with Prof. Cianchi.
\end{abstract}

\maketitle

\section{Introduction}

In this note, we want to give some remarks and improvements on a recent paper of us \cite{nos} about an optimization problem for the $p-$Laplace operator.

These remarks were motivated by some discussion the authors had with Prof. Cianchi and we are grateful to him.

Let us recall the problem analyzed in \cite{nos}.

Given a domain $\O\subset\rn$ (bounded, connected, with smooth
boundary) and some class of admissibel loads $\mathcal{A}$, in \cite{nos} we studied the following problem:
$$
\mathcal{J}(f):=\int_{\partial\Omega}f(x)u_f\,\rd\H^{N-1} \to \mbox{max}
$$
for $f\in \mathcal{A},$ where $\H^d$ denotes the $d-$dimensional Hausdorff measure and $u$ is the (unique) solution to
the nonlinear problem with load $f$
\begin{equation}\label{bati}
\begin{cases}
-\Delta_p u + |u|^{p-2} u = 0 & \textrm{in }  \O,\\
|\nabla u|^{p-2}\frac{\partial u}{\partial\nu}= f & \textrm{on }
\partial\O.
\end{cases}
\end{equation}
Where $p\in(1,\infty)$, $\Delta_p u = \mbox{div}(|\nabla u|^{p-2}\nabla u)$ is
the usual $p-$Laplacian, $\frac{\partial }{\partial\nu}$ is the outer normal
derivative and $f\in L^{q}(\partial \O)$ with $q>\frac{p'}{N'}$ .

In \cite{nos}, we worked with three different classes of admissible functions $\mathcal{A}$

\begin{itemize}
	\item The class of rearrangements of a given function $f_0.$
	\item The (unit) ball in some $L^q.$
	\item The class of characteristic functions of sets of given measure.	
\end{itemize}

For each of these classes, we proved existence of a maximizing load (in the respective class) and analyzed properties of
these maximizer.

When we worked in the unit ball of $L^{q},$ we explicitly found the (unique) maximizar for $\mathcal{J},$
namely, the first eigenfunction of a Steklov-like nonlinear eigenvalue problem.

Whereas when we worked with the class of characteristic functions of set of given boundary measure, 
besides to prove that there exists a maximizer function we could give a characterization of set where the maximizer 
function is supported. Moreover, in order to analyze properties of this maximizer, we computed the first variation 
with respect respect to perturbations on the set where the characteristic function was supported. See \cite{nos} (section 5).

The aim of this work is to generalize the results obtained for the class of characteristic functions of set of given 
boundary measure to the class of rearrangements function of a given function $f_0.$

Recall that if $f_0$ is a characteristic function of a set of $\H^{N-1}$-measure $\alpha$, then {\em every} characteristic function of a set of $\H^{N-1}$-measure $\alpha$ is a rearrangement of $f_0$.

\section{Characterization of Maximizer Function}

In this section we give characterization of the maximizer function relative to the class of rearrangements of a given 
function $f_0.$

We begin by observe that \eqref{bati}  has a unique weak solution $u_{f}$, for which the following equations hold
\begin{equation}\label{rojas}
\int_{\partial\O}fu_{f} \, \rd \H^{N-1} = \sup_{u\in W^{1,p}(\Omega)} \mathcal{I}(u),
\end{equation}
where
$$
\mathcal{I}(u):=\frac{1}{p-1}\Big\{ p\int_{\partial\O} f u \, \rd \H^{N-1} - \int_\O
|\nabla u|^p + |u|^p \, \rd \H^N \Big\}.
$$

Let $f_0\in\lqb$, with $q=p/(p-1),$ and let $\mathcal{R}_{f_0}$ be the class of
rearrangements of $f_0$. We was interested in finding
\begin{equation}\label{cani}
\sup_{f\in \mathcal{R}_{f_0}} \int_{\partial\O} fu_{f} \, \rd \H^{N-1}.
\end{equation}

In \cite{nos}, Theorem 3.1, we could proof that there exists $\hat{f}\in\mathcal{R}_{f_0}$ such that
$$
\int_{\partial\Omega}\hat{f}\hat{u} \, \textrm{d}\mathcal{H}^{N-1}=\sup_{f\in\mathcal{R}_{f_0}}\int_{\partial\Omega}
fu_{f}\,\textrm{d} \mathcal{H}^{N-1}.
$$
where $\hat{u}=u_{\hat{f}}.$

We begin by giving a characterization of this maximizer $\hat{f}$ in the spirit of \cite{cuccu}.

\begin{te}\label{teo1}
$\hat{f}$ is the unique maximizer of linear functional $L(f):=\int_{\partial\Omega}f\hat{u} \,\rd \mathcal{H}^{N-1},$ relative to $f\in \mathcal{R}_{f_0}.$ Therefore, there is an increasing function $\phi$ such that $\hat{f}=\phi\circ\hat{u}$ 
$\mathcal{H}^{N-1}-$a.e.
\end{te}

\begin{proof} We proceed in three steps.

\medskip

{\em Step 1.} First we show that $\hat{f}$ is a maximizer of $L(f)$ relative to $f\in \mathcal{R}_{f_0}.$

\smallskip

In fact, let $h\in \mathcal{R}_{f_0}$, since 
$\int_{\partial\Omega}\hat{f}\hat{u} \, \textrm{d}\mathcal{H}^{N-1}=\sup_{f\in\mathcal{R}_{f_0}}\int_{\partial\Omega}
fu_{f}\,\textrm{d} \mathcal{H}^{N-1},$
we have that
\begin{align*}
\int_{\partial\Omega}\hat{f}\hat{u} \, \textrm{d}\mathcal{H}^{N-1}&\ge\int_{\partial\Omega}
hu_{h}\,\textrm{d} \mathcal{H}^{N-1}\\
&= \sup_{u\in W^{1,p}(\Omega)}\frac{1}{p-1}\left\{p\int_{\partial\Omega}hu \, \rd \mathcal{H}^{N-1} - \int_{\partial\Omega}|\nabla u|^p 
+ |u|^p \, \rd \H^N\right\}\\
&\ge\frac{1}{p-1}\left\{p\int_{\partial\Omega}h\hat{u} \, \rd \mathcal{H}^{N-1} - \int_{\partial\Omega}|\nabla \hat{u}|^p 
+ |\hat{u}|^p \, \rd \H^N\right\},
\end{align*}
and, since
$$
\int_{\partial\Omega}\hat{f}\hat{u}\,\rd \mathcal{H}^{N-1}=\frac{1}{p-1}\left\{p\int_{\partial\Omega}\hat{f}\hat{u} \, \rd \mathcal{H}^{N-1} -
\int_{\partial\Omega}|\nabla \hat{u}|^p + |\hat{u}|^p \, 
\rd \H^N\right\},
$$
we have
$$
\int_{\partial\Omega}\hat{f}\hat{u} \,\rd \mathcal{H}^{N-1} \ge\int_{\partial\Omega}h\hat{u} \,\rd \mathcal{H}^{N-1}.
$$
Therefore,
$$
\int_{\partial\Omega}\hat{f}\hat{u} \,\rd \mathcal{H}^{N-1}=\sup_{f\in \mathcal{R}_{f_0}} L(f). 
$$

\medskip

{\em Step 2.} Now, we show that $\hat{f}$ is the unique maximizer of $L(f)$ relative to $f\in \mathcal{R}_{f_0}.$

\smallskip

We suppose that $g$ is another maximizer of $L(f)$ relative to $f \in \mathcal{R}_{f_0}$. Then
$$
\int_{\partial\Omega}\hat{f}\hat{u} \,\rd \mathcal{H}^{N-1}=\int_{\partial\Omega}g\hat{u} \,\rd \mathcal{H}^{N-1}.
$$
Thus
\begin{align*}
\int_{\partial\Omega}g\hat{u} \,\rd \mathcal{H}^{N-1}&=\int_{\partial\Omega}\hat{f}\hat{u} \,\rd\mathcal{H}^{N-1}\\
&\ge\int_{\partial\Omega}gu_g \,\rd \mathcal{H}^{N-1}\\
&=\sup_{u\in W^{1,p}(\Omega)}\frac{1}{p-1}\left\{p\int_{\partial\Omega}gu\, \rd
\mathcal{H}^{N-1}-\int_{\partial\Omega}|\nabla u|^p + |u|^p \, \rd \mathcal{H}^N\right\}.
\end{align*}

\noindent On the other hand,
\begin{align*}
\int_{\partial\Omega}g\hat{u} \,\rd \mathcal{H}^{N-1}&=\int_{\partial\Omega}\hat{f}\hat{u} \,\rd\mathcal{H}^{N-1}\\
&=\frac{1}{p-1}\left\{p\int_{\partial\Omega}\hat{f}\hat{u}\, \rd
\mathcal{H}^{N-1}-\int_{\partial\Omega}|\nabla \hat{u}|^p + |\hat{u}|^p \, \rd \mathcal{H}^N\right\}\\
&=\frac{1}{p-1}\left\{p\int_{\partial\Omega}g\hat{u}\, \rd
\mathcal{H}^{N-1}-\int_{\partial\Omega}|\nabla \hat{u}|^p + |\hat{u}|^p \, \rd \mathcal{H}^N\right\}.
\end{align*}
Then
$$
\int_{\partial\Omega}g\hat{u} \,\rd \mathcal{H}^{N-1}=\sup_{u\in W^{1,p}(\Omega)}\frac{1}{p-1}\left\{p\int_{\partial\Omega}gu\, \rd
\mathcal{H}^{N-1}-\int_{\partial\Omega}|\nabla u|^p + |u|^p \, \rd \mathcal{H}^N\right\}.
$$

\noindent Therefore $\hat{u}=u_g.$ Then $\hat{u}$ is the unique weak solution to 
$$
\begin{cases}
\Delta_p \hat{u} + |\hat{u}|^{p-2}\hat{u}=0 &\textrm{in } \Omega,\\
|\nabla \hat{u}|^{p-2}\frac{\partial \hat{u}}{\partial \nu}=g &\textrm{on } \partial\Omega.
\end{cases}
$$
Furthermore, we now that u is the unique weak solution to
$$
\begin{cases}
\Delta_p \hat{u} + |\hat{u}|^{p-2}\hat{u}=0 &\textrm{in } \Omega,\\
|\nabla \hat{u}|^{p-2}\frac{\partial \hat{u}}{\partial \nu}=\hat{f} &\textrm{on } \partial\Omega.
\end{cases}
$$
Therefor $\hat{f}=g$ $\mathcal{H}^{N-1}-$a.e.

\medskip

{\em Step 3.} Finally, we have that there is an increasing function $\phi$ such that 
$\hat{f}=\phi\circ\hat{u}$  $\mathcal{H}^{N-1}-$a.e.

\smallskip

This is a direct consequence of Steps 1, 2 and Theorem \ref{burton} below.

\medskip

This completes the proof of Theorem \ref{teo1}.
\end{proof}

In order to state Theorem \ref{burton}, we need the following definition

\begin{de}
The measure space $(X,\mathcal{M},\mu)$ is called nonatomic if for $U\in \mathcal{M}$ with $\mu(U)>0$, there exists $V\in \mathcal{M}$ with $V \subset U$ and $0<\mu(V)<\mu(U).$ The measure space $(X,\mathcal{M},\mu)$ is called separable if there is a sequence $\{U_n\}_{n=1}^{\infty}$ of measurable sets such that for every $V\in \mathcal{M}$ and $\varepsilon>0$ there exists n such that
$$
\mu(V\setminus U_n)+\mu(U_n\setminus V) <\varepsilon.
$$
\end{de}

\begin{te}[See \cite{b1}]
\label{burton}
Let $(X,\mathcal{M},\mu)$ be a finite separable nonatomic measure space, let $1\le p \le\infty,$ let $q$ be the conjugate exponent of $p,$ let $f_0\in L^p(X,\mu)$ and $g\in L^q(X,\mu)$ and let $R_{f_0}$ be the set of rearrangements of $f_0$ on $X.$ If $L(f)=\int_{X}fg \,\rd\mu$ has a unique maximizer $\hat{f}$ relative to $\mathcal{R}_{f_0}$ there is an increasing function $\phi$ such that $f^*=\phi\circ g$ $\mu-$a.e.
\end{te}

\section{Derivate with respect to the load}

Now we compute the derivate of the functional $\mathcal{J}(\hat{f})$ with respect to perturbations in $\hat{f}.$ We will consider regular perturbations and asume that the function $\hat{f}$ has bounded variation in $\partial\Omega.$

We begin by describing the kind of variations that we are considering. Let $V$ be a regular (smooth) vector field, globally 
Lipschitz, with support in a neighborhood of $\partial\O$ such that $\langle V,\nu \rangle=0$
and let $\psi_t:\rn\to\rn$ be defined as the unique solution to
\begin{equation}\label{moreno}
\begin{cases}
\frac{\rd}{\rd t}\psi_t (x)=V(\psi_t(x)) & t>0,\\
\psi_0(x)= x & x\in \rn.
\end{cases}
\end{equation}

We have
$$
\psi_t(x)=x+tV(x)+ o(t) \quad \forall x\in\rn.
$$

Thus, if $f\in \mathcal{R}_{f_0},$ we define $f_t =f\circ\psi_t^{-1}.$ Now, let
$$
I(t):= {\mathcal J}(f_t) = \int_{\partial\Omega}u_t f_t \textrm{d}\mathcal{H}^{N-1}
$$
where $u_t\in W^{1,p}(\Omega)$ is the unique solution to
\begin{equation}\label{balbo}
\begin{cases}
-\Delta_p u_t + |u_t|^{p-2} = 0 & \textrm{in }\Omega,\\
|\nabla u_t|^{p-2}\frac{\partial u_t}{\partial \nu}=f_t & \textrm{on } \partial\Omega.
\end{cases}
\end{equation}

\begin{lem}\label{con1} Given $f\in L^q(\partial\Omega)$ then
\begin{eqnarray*}
f_t=f\circ\psi_t^{-1} &\to& f \textrm{ in } L^{q}(\partial\Omega), \textrm{ as } t\to 0.
\end{eqnarray*}
\end{lem}

\begin{proof}
Let $\varepsilon>0,$ and let  $g\in C^{\infty}_{c}(\partial\Omega)$ fixed such that $\|f-g\|_{L^q(\partial\Omega)}<\varepsilon.$ By the usual change of variables formula, we have,
$$
\|f_t-g_t\|_{L^q(\partial\Omega)}^q=\int_{\partial\Omega}|f-g|^qJ_{\tau}\psi_t \rd\H^{N-1},
$$
where $g_t=g\circ \psi_t^{-1}$ and $J\psi$ is the tangential Jacobian of $\psi$. We also know that
$$
J_{\tau}\psi:=1+t \mbox{ div}_\tau V + o(t).
$$
Here $\textrm{div}_\tau V$ is the tangential divergence of $V$ over $\partial\Omega.$ Then
$$
\|f_t-g_t\|_{L^q(\partial\Omega)}^q=\int_{\partial\Omega}|f-g|^q(1 + t \mbox{ div}_\tau V +o(t))\,\rd\H^{N-1}.
$$
Then, there exist $t_1>0$ and such that if $0<t<t_1$ then
$$
\|f_t-g_t\|_{L^q(\partial\Omega)}\le C\varepsilon.
$$
where $C$ is a constant independent of $t.$ Moreover, since  $\psi_t^{-1}\to Id$ in the $C^1$ topology when $t\to0$ then 
	$g_t=g\circ\psi^{-1}_{t}\to g$ in the $C^1$ topology and therefore there exists	$t_2>0$  such that if $0<t<t_2$ then 
	$$\|g_t-g\|_{L^q(\partial\Omega)}<\varepsilon.$$
	Finally, we have for all $0<t<t_0=\min\{t_1,t_2\}$ then
	\begin{align*}
	\|f_t-f\|_{L^q(\partial\Omega)}&\le \|f_t-g_t\|_{L^q(\partial\Omega)}+\|g_t-g\|_{L^q(\partial\Omega)}
	+\|g-f\|_{L^q(\partial\Omega)}\\
	&\le C\varepsilon
        \end{align*}
        where $C$ is a constant independent of $t.$
\end{proof}

\begin{lem}\label{con2}
Let $u_0$ and $u_t$ be the solution of \eqref{balbo} with $t=0$ and $t>0$,
respectively. Then
       $$ u_t \to u_0 \textrm{ in } W^{1,p}(\Omega), \textrm{ as } t\to 0^+.$$
\end{lem}

\begin{proof}
The proof follows exactly as the one in Lemma 4.2 in \cite{cuccu}. The only difference being that we use the trace inequality instead of the Poincar\'e inequality.
\end{proof}

\begin{re}\label{remark.conver}
It is easy to see that, as $\psi_t\to Id$ in the $C^1$ topology, then from Lemma \ref{con2} it follows that
$$
w_t:=u_t\circ\psi_t\to u_0 \quad \textrm{strongly in } W^{1,p}(\Omega).$$
\end{re}

With these preliminaries, the following theorem follows exactly as Theorem 5.5 of \cite{nos}.

\begin{te}\label{derivada}
With the previous notation, we have that $I(t)$ is differentiable at $t=0$ and
\begin{align*}
\frac{\textrm{d}I(t)}{\textrm{d}t}\Big|_{t=0} 
& =  \frac{1}{p-1} \bigg\{p\int_{\partial\O} u_0 f \mbox{ div}_\tau V\, \rd\H^{N-1} + p \int_\O |\nabla u_0|^{p-2}\langle\nabla u_0, ^T V' \nabla u_0^T\rangle\, \rd \H^N\\ 
& \quad - \int_\O (|\nabla u_0|^p + |u_0|^p) \mbox{ div}\, V \,
\rd\H^N\bigg\}.
\end{align*}
where $u_0$ is the solution of \eqref{balbo} with $t=0.$
\end{te}

\begin{proof}
For the details see the proof of Theorem 5.5 of \cite{nos}.
\end{proof}

Now we try to find a more explicit formula for $I'(0)$. For This, we consider $f\in L^q(\partial\Omega)\cap 
BV(\partial\Omega)$, where $BV(\partial\Omega)$ is the space of functions of bounded variation. For details and properties of BV functions we refer to the book \cite{E-G}.

\begin{te}
If $f\in L^q(\partial \Omega)\cap BV(\partial\Omega),$ we have that
$$
\frac{\partial I(t)}{\partial t}\Big|_{t=0}=\frac{p}{p-1}\int_{\partial\Omega}u_0 V \,\rd[Df].
$$
where $u_0$ is the solution of \eqref{balbo} with $t=0.$
\end{te}

\begin{proof}
In the course of the computations, we require the solution $u_0$ to
$$
\begin{cases}
-\Delta u_0 + |u_0|^{p-2} u_0 = 0 & \mbox{in }\Omega,\\
|\nabla u_0|^{p-2}\frac{\partial u_0}{\partial \nu} = f & \mbox{on }
\partial\Omega,
\end{cases}
$$
to be $C^2$. However, this is not true. As it is well known (see, for instance,
\cite{16}), $u_0$ belongs to the class $C^{1,\delta}$ for some $0<\delta<1$.

In order to overcome this difficulty, we proceed as follows. We consider the
regularized problems
\begin{equation}\label{regularized}
\begin{cases}
-\mbox{div}( (|\nabla u_0^\varepsilon|^2 + \varepsilon^2)^{(p-2)/2}\nabla u^\varepsilon_0) + |u^\varepsilon_0|^{p-2} 
u^\varepsilon_0 = 0 & \mbox{in }\Omega,\\
(|\nabla u^\varepsilon_0|^2 + \varepsilon^2)^{(p-2)/2}\frac{\partial u^\varepsilon_0}{\partial \nu} =
f & \mbox{on }\partial\Omega.
\end{cases}
\end{equation}
It is well known that the solution $u_0^\varepsilon$ to \eqref{regularized} is of class
$C^{2,\rho}$ for some $0<\rho<1$ (see \cite{LSU}).

Then, we can perform all of our computations with the functions $u_0^\varepsilon$ and pass to the limit as $\varepsilon\to 0+$ at the end.

We have chosen to work formally with the function $u_0$ in order to make our arguments more transparent and leave the details to the reader. For a similar approach, see \cite{GM}.

Now, by Theorem \ref{derivada} and since
\begin{eqnarray*}
\mbox{div}(|u_0|^pV) & = & p|u_0|^{p-2}u_0\langle\nabla u_0, V\rangle + |u_0|^p \mbox{ div}\, V,\\
\mbox{div}(|\nabla u_0|^pV)&=&p|\nabla u_0|^{p-2}\langle\nabla u_0 D^2 u_0, V\rangle + |\nabla u_0|^p \mbox{ div}\, V,
\end{eqnarray*}
we obtain
\begin{align*}
I'(0) = & \frac{1}{p-1} \bigg\{p\int_{\partial\O} u_0 f \mbox{ div}_\tau V\,
\rd\H^{N-1} + p \int_\O |\nabla u_0|^{p-2}\langle\nabla u_0, ^T V' \nabla u_0^T\rangle\, \rd \H^N\\
& - \int_\O (|\nabla u_0|^p + |u_0|^p) \mbox{ div}V \, \rd\H^N\\
= & \frac{1}{p-1} \bigg\{p\int_{\partial\O} u_0 f \mbox{ div}_\tau
V\,\rd\H^{N-1} + p\int_\O |\nabla u_0|^{p-2}\langle\nabla u_0, ^T V' \nabla u_0^T\rangle\, \rd \H^N\\
& -\int_\O \mbox{div}((|\nabla u_0|^p + |u_0|^p)V)\, \rd\H^N
+ p\int_\O |\nabla u_0|^{p-2}\langle \nabla u_0 D^2 u_0, V\rangle\, \rd\H^N\\
&  + p\int_\O |u_0|^{p-2}u_0 \langle \nabla u_0, V\rangle \, \rd \H^N\bigg\}.
\end{align*}
Hence, using that $\langle V, \nu\rangle=0$ in the right hand side of the above
equality we find
\begin{eqnarray*}
I'(0) & = & \frac{p}{p-1}\bigg\{\int_{\partial\O} u_0 f \mbox{ div}_\tau V\, \rd\H^{N-1}\\
& & + \int_\O |\nabla u_0|^{p-2} \langle \nabla u_0, ^T V' \nabla u_0^T + D^2 u_0 V^T\rangle\, \rd \H^N\\
& & + \int_\O |u_0|^{p-2}u_0 \langle \nabla u_0, V\rangle \, \rd \H^N\bigg\}\\
& = & \frac{p}{p-1} \bigg\{\int_{\partial\O} u_0 f \mbox{ div}_\tau V\,
\rd\H^{N-1} + \int_\O |\nabla u_0|^{p-2} \langle \nabla u_0, \nabla(\langle\nabla u_0, V\rangle)\rangle\, \rd \H^N\\
& & + \int_\O |u_0|^{p-2}u_0 \langle \nabla u_0, V\rangle \, \rd \H^N\bigg\}.
\end{eqnarray*}
Since $u_0$ is a week solution of \eqref{balbo} with $t=0$ we have
\begin{eqnarray*}
I'(0) & = & \frac{p}{p-1}\bigg\{\int_{\partial\O} u_0  f \mbox{ div}_\tau
V\, \rd\H^{N-1} + \int_{\partial\O} \langle \nabla u_0, V\rangle f\, \rd\H^{N-1}\bigg\}\\
& = & \frac{p}{p-1} \int_{\partial\O} \mbox{div}_\tau(u_0 V) f\, \rd\H^{N-1}\\
\end{eqnarray*}
Finally, since $f\in BV(\partial\Omega)$ and $V\in C^1(\partial\Omega;\rn),$
\begin{eqnarray*}
I'(0)& = & \frac{p}{p-1} \int_{\partial\O} \mbox{div}_\tau(u_0 V) f\, \rd\H^{N-1}\\
&=&\frac{p}{p-1}\int_{\partial\Omega}u_0 V \,\rd[Df].
\end{eqnarray*}
The proof is now complete.
\end{proof}

\end{document}